\documentclass[12pt]{amsart}

\usepackage{amssymb}
\usepackage{amsmath}
\usepackage{amsfonts}
\usepackage{graphicx}

\usepackage[total={17cm,22cm},top=2.5cm, left=2.3cm]{geometry}
\usepackage{hyperref}
    \usepackage{aeguill}
    \usepackage{type1cm}

\usepackage[shortlabels]{enumitem}

\newtheorem{thm}{Theorem}
\newtheorem{cor}[thm]{Corollary}
\newtheorem{lem}[thm]{Lemma}

\newtheorem{rem}[thm]{Remark}
\newtheorem{ex}[thm]{Example}

\newtheorem{claim}[thm]{Claim}

\usepackage{amsxtra}
\usepackage{eucal}
\usepackage{mathrsfs}
\usepackage{longtable}
\usepackage{caption}

\usepackage{hyperref}
    \usepackage{aeguill}
    \usepackage{type1cm}

\newcommand{\N}{\mathbb{N}}
\newcommand{\Z}{\mathbb{Z}}

\newcommand{\R}{\mathbb{R}}

\usepackage[usenames,dvipsnames]{color}

\usepackage[dvipsnames]{xcolor}

\begin{document}

\title{Subdifferentiable functions satisfy Lusin properties of class $C^1$ or $C^2$}

\author{D. Azagra}
\address{ICMAT (CSIC-UAM-UC3-UCM), Departamento de An{\'a}lisis Matem{\'a}tico,
Facultad Ciencias Matem{\'a}ticas, Universidad Complutense, 28040, Madrid, Spain}
\email{azagra@mat.ucm.es}

\author{J. Ferrera}
\address{IMI, Departamento de An{\'a}lisis Matem{\'a}tico,
Facultad Ciencias Matem{\'a}ticas, Universidad Complutense, 28040, Madrid, Spain}
\email{ferrera@mat.ucm.es}

\author{M. Garc\'ia-Bravo}
\address{ICMAT  (CSIC-UAM-UC3-UCM), Calle Nicol\'as Cabrera 13-15.
28049 Madrid, Spain}
\email{miguel.garcia@icmat.es}

\author{J. G\'omez-Gil}
\address{Departamento de An{\'a}lisis Matem{\'a}tico,
Facultad Ciencias Matem{\'a}ticas, Universidad Complutense, 28040, Madrid, Spain}
\email{gomezgil@mat.ucm.es}


\begin{abstract}
Let $f:\R^n\to\R$ be a function. Assume that for a measurable set
$\Omega$  and almost every $x\in\Omega$ there exists a vector
$\xi_x\in\R^n$ such that $$\liminf_{h\to 0}\frac{f(x+h)-f(x)-\langle
  \xi_x, h\rangle}{|h|^2}>-\infty.$$ Then we show that $f$ satisfies a
Lusin-type property of order $2$ in $\Omega$, that is to say, for
every $\varepsilon>0$ there exists a function $g\in C^2(\R^n)$ such
that $\mathcal{L}^{n}(\{x\in\Omega : f(x) \neq g(x)\} \leq\varepsilon$. In particular every function which has a nonempty proximal subdifferential almost everywhere also has the Lusin property of class $C^2$. We also obtain a similar result (replacing $C^2$ with $C^1$) for the Fr\'echet subdifferential. Finally we provide some examples showing that this kind of results are no longer true for {\em Taylor subexpansions} of higher order.
\end{abstract}

\maketitle

A classical theorem of Lusin \cite{Lusin} states that for every Lebesgue measurable function $f:\R^n\to\R$ and every $\varepsilon>0$ there exists a continuous function $g:\R^n\to\R$ such that 
\begin{equation}\label{LusinProperty}
\mathcal{L}^n\left(\{x\in \R^n : f(x)\neq g(x)\}\right)\leq\varepsilon.
\end{equation}
Here, as in the rest of this note, $\mathcal{L}^{n}$ denotes the Lebesgue measure in $\R^n$.

Several authors have shown that one can take $g$ of class $C^k$, provided that $f$ has some regularity properties of order $k$ (for instance, locally bounded distributional derivatives up to the order $k$, or Taylor expansions of order $k$ almost everywhere). If, given a differentiability class $\mathcal{C}$ and a function $f:\R^n\to\R$ we can find, for each $\varepsilon>0$, a function $g\in\mathcal{C}$ satisfying \eqref{LusinProperty}, we will say that $f$ {\em has the Lusin property of class $\mathcal{C}$}.

The first of such results was discovered by Federer \cite[p. 442]{Federer}, who showed that a.e differentiable functions (and in particular locally Lipschitz functions) have the Lusin property of class $C^1$. H. Whitney \cite{Whitney2} improved this result by showing that a function $f:\R^n\to\R$ has approximate partial derivatives of first order a.e. if and only if $f$ has the Lusin property of class $C^1$. 

In \cite[Theorem 13]{CalderonZygmund} Calderon and Zygmund established analogous results of order $k$ for the classes of Sobolev functions $W^{k,p}(\R^n)$. Other authors, including Liu \cite{Liu1977}, Bagby, Michael and Ziemer \cite{BagbyZiemer, MichaelZiemer, Ziemer}, Bojarski, Haj\l asz and Strzelecki \cite{BojarskiHajlasz, BojarskiHajlaszStrzelecki}, and Bourgain, Korobkov and Kristensen \cite{BourKoKris2} have improved Calderon and Zygmund's result in different ways, by obtaining additional estimates for $f-g$ in the Sobolev norms, as well as the Bessel capacities or the Hausdorff contents of the exceptional sets where $f\neq g$. In \cite{BourKoKris2} some Lusin properties of the class $BV_{k}(\R^n)$ (of integrable functions whose distributional derivatives of order up to $k$ are Radon measures) are also established. The Whitney extension technique \cite{Whitney}, and some related techniques as the {\em Whitney smoothing} introduced in \cite{BojarskiHajlaszStrzelecki},  play a key role in the proofs of all of these results. 

For the special class of convex functions $f:\R^n\to\R$, Alberti and Imonkulov \cite{Alberti2, Imonkulov} showed that every convex function has the Lusin property of class $C^2$ (with $g$ not necessarily convex in \eqref{LusinProperty}); see also \cite{Alberti} for a related problem. More recently Azagra and Haj\l asz \cite{AzagraHajlasz} have proved that $g$ can be taken to be $C^1$ and convex in \eqref{LusinProperty} if and only if either $f$ is essentially coercive (meaning that $f$ is coercive up to a linear perturbation) or else $f$ is already $C^1$ (in which case taking $g=f$ is the only possible option); they have also shown that if $f:\R^n\to\R$ is strongly convex then for every $A\subset\R^n$ of finite measure and every $\varepsilon>0$ there exists $g:\R^n\to\R$ convex and $C^{1, 1}$ such that $\mathcal{L}^n\left(\{x\in A : f(x)\neq g(x)\}\right)\leq\varepsilon$.

On the other hand, generalizing Whitney's result \cite{Whitney2} to higher orders of differentiability, Isakov \cite{Isakov} and Liu and Tai \cite{LiuTai} independently established that a function $f:\R^n\to\R$ has the Lusin property of class $C^{k}$ if and only if $f$ is approximately differentiable of order $k$ almost everywhere (and if and only if $f$ has an approximate $(k-1)$-Taylor polynomial at almost every point).

In this note we will answer the following question (which we think may be quite natural for people working on nonsmooth analysis or viscosity solutions to PDE such as Hamilton-Jacobi equations): do functions with nonempty subdifferentials a.e. have Lusin properties of order $C^1$ or $C^2$? By subdifferentials we mean the Fr\'echet subdifferential, or the proximal subdifferential, or the second order viscosity subdifferential; see \cite{ClarkeEtAl, CIL, FerreraBook} and the references therein for information about subdifferentials and their applications. As we will see the answer is positive: Fr\'echet subdifferentiable functions have the Lusin property of class $C^1$, and functions with nonempty proximal subdifferentials a.e. (in particular functions with a.e. nonempty viscosity subdifferentials of order $2$) have the Lusin property of class $C^2$. 

This question can be formulated in a more general form (perhaps appealing to a wider audience) as a problem about {\em Taylor subexpansions}: given $k\in\N$ and a function $f:\R^n\to\R$, assume that for almost every $x\in\R^n$ there exists a polynomial $P_x$ of degree less than or equal to $k-1$ such that
$$
\liminf_{y\to x}\frac{f(y)-P_{x}(y)}{|y-x|^{k}}>-\infty.
$$
Is it then true that $f$ has the Lusin property of order $k$? 

The results of this note will show that the answer to this question is positive for $k=1, 2$, but negative for $k\geq 3$.

In the case $k=1$ the proof is very simple and natural.

\begin{thm}\label{Theorem for k=1}
Let $\Omega\subset\R^n$ be a Lebesgue measurable set, and $f:\Omega\rightarrow\R$ a function. Assume that for almost every $x\in\Omega$ we have
\begin{equation}\label{1}
\liminf_{y\rightarrow x}\frac{f(y)-f(x)}{|y-x|}>-\infty.
\end{equation}
Then, for every $\varepsilon>0$ there exists a function $g\in C^{1}(\R^n)$ such that
$$
\mathcal{L}^n\left(\{x\in \Omega : f(x)\neq g(x)\}\right)\leq\varepsilon.
$$
\end{thm}
In order to facilitate the proof of Theorem \ref{Theorem for k=1}, as well as that of Theorem \ref{Theorem for k=2} below, let us state the following technical lemma, which is standard. We include its proof for the readers' convenience.
\begin{lem}\label{Lusin on compacta is enough}
Let $\Omega$ be a Lebesgue measurable subset of $\R^n$, $k\in\N$, and $f:\Omega\to\R$ be measurable. Then $f$ has the Lusin property of class $C^k$ (meaning that for every $\varepsilon>0$ there exists $g\in C^k(\R^n)$ such that $\mathcal{L}^{n}\left\{x\in\Omega : f(x)\neq g(x)\}\right)\leq\varepsilon$) if and only if the restriction of $f$ to each compact subset of $\Omega$ has the Lusin property of class $C^k$.
\end{lem}
\begin{proof}
It is obvious that if $f:\Omega\to\R$ has the Lusin property of class $C^k$ then, for every compact subset $K$ of $\Omega$, the function $f_{|_K}:K\to\R$ has the Lusin property of class $C^k$. Let us prove the converse. Assume first that $\Omega$ is bounded. By the regularity of the measure $\mathcal{L}^{n}$, for every $\varepsilon>0$ we may find $K_{\varepsilon}$, a compact subset of $\Omega$, such that
$\mathcal{L}^{n}\left(\Omega\setminus K_{\varepsilon}\right)\leq\varepsilon/2$. By assumption, there exists a function $g=g_{K_{\varepsilon}}\in C^k(\R^n)$ such that $\mathcal{L}^{n}\left(\{x\in K_{\varepsilon}: f(x)\neq g(x)\}\right)\leq\varepsilon/2$. Then we have
$$
\mathcal{L}^{n}\left\{x\in\Omega : f(x)\neq g(x)\}\right)\leq
\mathcal{L}^{n}\left(\Omega\setminus K_{\varepsilon}\right)+\mathcal{L}^{n}\left(\{x\in K_{\varepsilon}: f(x)\neq g(x)\}\right)\leq\varepsilon,
$$
and therefore $f:\Omega\to\R$ has the Lusin property of class $C^k$. 

Now let us consider the general case that $\Omega$ is not necessarily bounded. We can write
$$
\Omega=\bigcup_{j=1}^{\infty}\Omega_j, \textrm{ where } \Omega_1=\Omega\cap \textrm{int} B(0,1), \textrm{ and } \Omega_{j+1} :=\Omega\cap \textrm{int} B(0, j+1)\setminus B(0, j),
$$
where $B(x,r)$ denotes the closed ball of center $x$ and radius $r$.
According to the previous argument, for each $j\in\N$ there exists a function $g_{j}\in C^{k}(\R^n)$ such that
$$
\mathcal{L}^{n}\left(\{x\in \Omega_{j}  : g_{j}(x)\neq f(x)\}\right)\leq \frac{\varepsilon}{6^{j}}.
$$
Let $(\psi_{j})_{j=1}^{\infty}$ be a $C^{\infty}$ smooth partition of unity subordinated to the covering $\{\textrm{int} B(0, j+1)\setminus B(0, j-1)\}_{j=1}^{\infty}\cup \{\textrm{int} B(0,1)\}$ of $\R^n$ (see for instance \cite[Ch. 2, Theorem 2.1]{Hirsch}), and let us define
$$
g(x)=\sum_{j=1}^{\infty}\psi_{j}(x)g_{j}(x).
$$
Notice that
$$
\{x\in\Omega_{j} : f(x)\neq g(x)\}\subseteq \bigcup_{i=j-1}^{j}\{x\in \Omega_j : f(x)\neq g_{i}(x)\}.
$$
This implies that
$$
  \mathcal{L}^{n}\left(\{x\in\Omega : f(x)\neq g(x)\}\right)\leq
  2\sum_{j=1}^{\infty}\mathcal{L}^{n}\left( \{x\in \Omega_{j} :
    f(x)\neq g_{j}(x)\}\right) \leq
  2\sum_{j=1}^{\infty}\frac{\varepsilon}{6^{j}}\leq \varepsilon,
$$
and concludes the proof of the Lemma.
\end{proof}

Now let us present the proof of Theorem \ref{Theorem for k=1}.
Let us call $N\subset\Omega$ the set of points for which \eqref{1} does not hold. Since $N$ has measure zero, proving Lusin property of class $C^1$ for the restriction of $f$ to $\Omega\setminus N$ would immediately lead to Lusin property of class $C^1$ for $f$. So we may and do assume in what follows that $N=\emptyset$, and in particular that
$$
\liminf_{y\rightarrow x}\frac{f(y)-f(x)}{|y-x|}>-\infty
$$
for every $x\in\Omega$. Note that this inequality implies that $f$ is lower semicontinuous on $\Omega$, and in particular $f$ is measurable. Now, according to Lemma \ref{Lusin on compacta is enough}, it is enough to check that the restriction of $f$ to every compact subset of $\Omega$ has the Lusin property of class $C^1$, and therefore we may also assume without loss of generality that $\Omega$ is compact. 
Define for each $j\in\N$,
$$
  E_j:=\left\lbrace x\in\Omega:\,f(y)-f(x)\geq -j|y-x|\;\text{for
      all}\; y\in B\left( x,\frac{1}{j}\right)\cap\Omega \right\rbrace
  \cap\left\lbrace x\in\Omega:\,|f(x)|\leq j\right\rbrace.
$$
Because $f$ is lower semicontinuous the sets $$\left\lbrace x\in\Omega:\,f(y)-f(x)\geq -j|y-x|\;\text{for all}\; y\in B\left( x,\frac{1}{j}\right)\cap\Omega \right\rbrace$$ are closed, and by using the measurability of $f$ this implies that each set $E_j$ is measurable. These sets form an increasing sequence such that
$$
\Omega=\bigcup^{\infty}_{j=1}E_j,
$$
so we have
$$
\lim_{j\to\infty}\mathcal{L}^{n}\left(\Omega\setminus E_j\right)=0,
$$
and therefore, for a given $\varepsilon>0$ we may find $j_0\in\N$ large enough such that 
$\mathcal{L}^n(\Omega\setminus E_{j_0})<\frac{\varepsilon}{2}$.

Take now $x,y\in E_{j_0}$. If $|y-x|\leq \frac{1}{j_0}$ then we have
$$
|f(y)-f(x)|\leq j_0|y-x| \textrm{ and } |f(x)|\leq j_0 .
$$   
On the other hand, if $x,y\in E_{j_0}$ and $|y-x|>1/j_0$ then we trivially get
$$
|f(y)-f(x)|\leq  2\sup_{z\in E_{j_0}}|f(z)|\leq M_0 |y-x|,
$$
where $M_0 :=2j_0\left( 1+ \sup_{z\in\Omega}|f(z)|\right)$. 

Observe that $M_0\geq j_0$. Thus in either case we see that
$$
|f(y)-f(x)|\leq M_0 |y-x| \textrm{ and } |f(x)|\leq M_0, \textrm{ for all } x, y\in E_{j_0}.
$$
That is, $f$ is bounded and $M_0$-Lipschitz on $E_{j_0}$. Then we can extend $f$ to a Lipschitz function $F$ on $\R^n$, for instance by using the McShane-Whitney formula
$$
F(x)=\inf_{y\in E_{j_0}}\{f(y)+M_0 |x-y|\},
$$ which defines an $M_0$-Lipschitz function on $\R^n$ that coincides with $f$ on $E_{j_0}$.
Obviously we have
$$\mathcal{L}^n(\left\lbrace x\in\Omega:\,f(x)\neq F(x)\right\rbrace ) \leq \mathcal{L}^n(\Omega\setminus E_{j_0})<\frac{\varepsilon}{2}.$$
But according to the result of Federer's that we mentioned above (see also \cite[Theorem 6.11]{EvansGariepy}), Lipschitz functions have the $C^1$ Lusin property, so we  may find another function $g\in C^1(\R^n)$ such that $\mathcal{L}^n(\left\lbrace x\in\Omega:\,F(x)\neq g(x)\right\rbrace )<\frac{\varepsilon}{2}$. Thus we conclude that
\begin{multline*}
\mathcal{L}^n(\left\lbrace x\in\Omega:\,f(x)\neq g(x)\right\rbrace )=
\\ = \mathcal{L}^n(\left\lbrace x\in E_{j_0}:\,F(x)\neq g(x)\right\rbrace \cup \left\lbrace x\in\Omega\setminus E_{j_0}:\,f(x)\neq g(x)\right\rbrace)\leq\\
\leq \mathcal{L}^n(\left\lbrace x\in E_{j_0}:\,F(x)\neq g(x)\right\rbrace)+\mathcal{L}^n(\Omega\setminus E_{j_0})\leq \frac{\varepsilon}{2}+\frac{\varepsilon}{2}=\varepsilon.
\end{multline*}
\qed

\begin{cor}\label{corollary for the Frechet subdifferential}
Let $U$ be a measurable subset of $\R^n$, $f:U\to\R$ be a measurable function, and define $\Omega=\{x\in U : D^{-}f(x)\neq\emptyset\}$. Then for every $\varepsilon>0$ there exists a function $g\in C^{1}(\R^n)$ such that
$$
\mathcal{L}^{n}\left(\{x\in\Omega : f(x)\neq g(x) \}\right)\leq\varepsilon.
$$
\end{cor}
Here $D^{-}f(x)$ denotes the Fr\'echet subdifferential of $f$ at $x$, that is the set of vectors $\zeta\in\R^n$ such that
$$
\liminf_{h\to 0}\frac{f(x+h)-f(x)-\langle \zeta, h\rangle}{|h|}\geq 0.
$$

\begin{rem}\label{the gradients equal the subgradients ae}
In the above corollary we also have $D^{-}f(x)=\{\nabla g(x)\}$ for almost every $x\in\Omega$ with $f(x)=g(x)$.
\end{rem}
\begin{proof}
Almost every point of the set $A=\{x\in\Omega : f(x)=g(x)\}$ is a point of density $1$ of $A$, and for every such point $x$ and every $\xi_x\in D^{-}f(x)$ we have
\begin{equation*}
0\leq \liminf_{y\to x, y\in A}\frac{f(y)-f(x)-\langle \xi_{x}, y-x\rangle}{|y-x|}=\liminf_{y\to x, y\in A}\frac{g(y)-g(x)-\langle \xi_{x}, y-x\rangle}{|y-x|},
\end{equation*}
and
\begin{equation*}
\lim_{y\to x, y\in A}\frac{g(y)-g(x)-\langle \nabla g(x), y-x\rangle}{|y-x|}=0,
\end{equation*}
hence also
\begin{eqnarray}\label{we are going to have the gradient of g at x equal xi sub x}
\liminf_{y\to x, y\in A}\frac{\langle \nabla g(x)-\xi_x, y-x\rangle}{|y-x|}\geq 0,
\end{eqnarray}
which, because $x$ is a point of density $1$ of $A$ and $h\mapsto \langle \nabla g(x)-\xi_x, h\rangle$ is linear, implies that $\nabla g(x)=\xi_x$. Indeed, we have 
\begin{equation}\label{A is of density 1}
\lim_{r\to 0^{+}}\frac{\mathcal{L}^{n}\left(A\cap B(x, r)\right)}{\mathcal{L}^{n}\left( B(x, r)\right)}=1.
\end{equation}
Assume we had $\zeta:=\nabla g(x)-\xi_x\neq 0$, and consider the sets $$S_{\zeta}:=\{v\in\R^n \, : \, |v|=1, \, \langle\zeta,v\rangle\leq-\frac{1}{2}|\zeta|\},$$ which determines a region of positive surface measure in the unit sphere, and the associated cone
$$
C_{x, \zeta}=\{x+tv \, : \, v\in S_{\zeta}, t>0\},
$$
of which $x$ is thus a point of positive density. Hence $C_{x, \zeta}$ also satisfies, in view of \eqref{A is of density 1}, that
$$
\liminf_{r\to 0^{+}}\frac{\mathcal{L}^{n}\left(A \cap C_{x, \zeta}\cap B(x, r)\right)}{\mathcal{L}^{n}\left( B(x, r)\right)}>0.
$$
In particular there exists a sequence $(y_k)=(x+t_kv_k)\subset A\cap C_{x, \zeta}$ (with $t_k>0$ and $v_k\in S_{\zeta}$, $k\in\N$) such that $\lim_{k\to\infty}y_k=x$. For this sequence we have, because of the definition of $C_{x, \zeta}$, that
$$
\frac{\langle\nabla g(x)-\xi_x, y_k-x\rangle}{|y_k-x|} =\frac{\langle\zeta, t_k v_k\rangle}{t_k}\leq -\frac{1}{2}|\zeta|<0
$$
for all $k\in\N$, which contradicts \eqref{we are going to have the gradient of g at x equal xi sub x}.
\end{proof}

A natural question at this point is the following. Does Corollary \ref{corollary for the Frechet subdifferential} hold true if we replace the Frechet subdifferential by the {\em limiting subdifferential}?
Let us recall that the limiting subdifferential  $\partial_Lf(x)$ of a lower semicontinuous function 
$f:\R^n\to\R$ at a point $x$ consists of all vectors of the form $\zeta =\lim_n\zeta _n$, where $\zeta _n\in D^-f(x_n)$, 
for sequences $\{x_n\}$ satisfying $\lim_nx_n=x$, and $\lim_nf(x_n)=f(x)$;  see \cite{ClarkeEtAl, FerreraBook}, for instance, for elementary properties of this subdifferential.
The question is whether or not the assumption that $\partial _Lf(x)\not= \emptyset$ for every
$x\in \R^n$ implies that $f$ satisfies the Lusin property of order $C^1$. Since one trivially has that $D^-f(x)\subset \partial _Lf(x)$, such a result would be much stronger than Corollary \ref{corollary for the Frechet subdifferential} above. The following example shows that the answer is negative.

\begin{ex}\label{example Takagi}
{\em
We consider the classical Takagi function $T:\R \to \R$ defined as follows. If $D_n$ denotes
the set  of real numbers $\{\frac{k}{2^n}: k\in \Z \}$, and $d(x,D_n)$ is the distance of $x$
to $D_n$, then
$$
T(x)=\sum_{n=1}^{\infty}d(x,D_n)
$$
This function was introduced by Takagi, \cite{Takagi}, as an easy example of a continuous function which is nowhere differentiable.  In \cite[Theorem 2]{BrownKozlowski} it is proved that $T$ does not agree with any $C^1$ function on any set of positive measure, and in particular $T$  does not
satisfy the Lusin property of order $C^1$. However, in \cite[Corollary 1.4]{FerreraGomezGil}, and also implicitly in \cite{GoraStern}, it is proved that $\partial _LT(x)=\R$ for every $x\in \R$.
}
\end{ex}

Concerning the Lusin property of class $C^2$ we have the following result.

\begin{thm}\label{Theorem for k=2}
Let $\Omega\subset\R^n$ be a Lebesgue measurable set, and $f:\Omega\rightarrow\R$ be a function such that for almost every $x\in\Omega$ there exists a vector $\xi_x\in\R^n$ such that
\begin{equation}\label{proximal limit}
\liminf_{y\rightarrow x}\frac{f(y)-f(x)-\langle\xi_x,y-x\rangle}{|y-x|^2}>-\infty.
\end{equation}
Then for every $\varepsilon>0$ there exists a function $g\in C^{2}(\R^n)$ such that
$$
\mathcal{L}^{n}\left(\{x\in\Omega : f(x)\neq g(x)\}\right)\leq\varepsilon.
$$
\end{thm}
\begin{proof}
Let $N$ be the subset of points for which \eqref{proximal limit} does not hold, and put $\Omega_1=\Omega\setminus N$. Since $N$ has measure zero, it will be enough to show that the restriction $f_1$ of $f$ to $\Omega_1$ has the Lusin property of class $C^2$. Since \eqref{proximal limit} holds for every $x\in\Omega_1$, it follows that $f$ is lower semicontinuous on $\Omega_1$, and in particular $f_1$ is measurable (hence so is $f$, since $N$ has measure zero). Now, according to Lemma \ref{Lusin on compacta is enough}, if we take an arbitrary compact subset $\Omega_2$ of $\Omega_1$, it will be enough for us to check that the restriction $f_2$ of $f_1$ to $\Omega_2$ has the Lusin property of class $C^2$. 

Because \eqref{proximal limit} holds for every $x\in\Omega_2$ and this implies
$$
\liminf_{y\rightarrow x} \frac{f_2(y)-f_2(x)}{|y-x|}>-\infty
$$
for all $x\in\Omega_2$, given $\varepsilon>0$, we may apply Theorem \ref{Theorem for k=1} to get a function $g\in C^1(\R^n)$ such that
$$
\mathcal{L}^n(\left\lbrace x\in\Omega_ 2 \, : \, f_2(x)\neq g(x)\right\rbrace )\leq\frac{\varepsilon}{4}.
$$
Observe also that the set  $A=\{x\in\Omega_2 : f_2(x)=g(x)\}$ is measurable and bounded, and according to the preceding remark we have $\xi_x=\nabla g(x)$ for almost every $x\in A$, so we can find a compact subset $\Omega_3$ of $A$ such that $\mathcal{L}^{n}\left(A\setminus \Omega_3\right)\leq\varepsilon/4$ and $\xi_x=\nabla g(x)$ for all $x\in\Omega_3$. Then we have that
\begin{equation}\label{proximal limit for }
\liminf_{y\rightarrow x, y\in \Omega_3}\frac{g(y)-g(x)-\langle\nabla g(x),y-x\rangle}{|y-x|^2}>-\infty
\end{equation}
for every $x\in\Omega_3$.
Now let us define for each $j\in \N$
$$
E_j:=\left\lbrace x\in \Omega_3 \, : \,g(y)-\langle \nabla g(x),y\rangle\geq g(x)-\langle \nabla g(x),x\rangle-j|y-x|^2\;\text{for all}\; y\in \Omega_3 \right\rbrace,
$$
and note that the sets $E_j$ are measurable and increasing to $\Omega_3$. There exists $j_0\in\N$ such that
$$\mathcal{L}^n(\Omega_3\setminus E_{j_0})\leq \frac{\varepsilon}{4}.$$

It will be enough for us to prove the following:
\begin{claim}\label{4}
We have that
$$
\limsup_{ y\rightarrow x,\;y\in E_{j_0}} \frac{|g(y)-g(x)-\langle \nabla g(x),y-x\rangle|}{ |y-x|^2}<+\infty
$$
for almost every $x\in E_{j_0}$.
\end{claim}
Assume for a moment that the Claim is true, that is, the restriction of  $g$ to $E_{j_0}$ has an approximate $(2-1)$-Taylor polynomial at every $x\in E_{j_0}$. By \cite[Theorem 1]{LiuTai} this is equivalent to saying that the restriction of $g$ to $E_{j_0}$ has the Lusin property of class $C^2$. So we may find a function  $h\in C^2(\R^n;\R)$ such that
$$
\mathcal{L}^n(\left\lbrace x\in E_{j_0} \, ; \, g(x)\neq h(x)\right\rbrace )\leq \frac{\varepsilon}{4},
$$
and we easily conclude that
$$
\mathcal{L}^n(\left\lbrace x\in \Omega_2:\,f_2(x)\neq h(x)\right\rbrace )\leq\varepsilon,
$$
as we wanted to show.

In order to prove Claim \eqref{4} we will borrow some ideas from \cite{KocanWang}.  We define new functions $\widetilde{g}:\R^n\to\R$ and $\hat{g}:\R^n\to\R$ by
$$
\begin{array}{lc} 
\widetilde{g}(x)=g(x)+j_0|x|^2 ,&\; x \in \R^n \\  \hat{g}(x)=\sup\left\lbrace p(x):\, p\;\text{affine and}\;p\leq \widetilde{g}\text{ on } \Omega_3\;\right\rbrace ,&\;  x\in \R^n   
                                                \end{array}
                                                $$
By definition of $E_{j_0}$ we have $\widetilde{g}(y)\geq \widetilde{g}(x)+\langle\nabla\widetilde{g}(x),y-x\rangle\;\text{for all}\; y\in \Omega_3, x\in E_{j_0}$, and by using this inequality it is easy to see that
$$\widetilde{g}(x)=\hat{g}(x)$$
for all $x\in E_{j_0}$. On the other hand, since $\Omega_3$ is compact and $g$ is continuous on $\Omega_3$, it is easy to see that $\hat{g}$ is everywhere finite. Moreover, as a supremum of affine functions, $\hat{g}$ is convex. Therefore $\hat{g}$ is locally Lipschitz on $\Omega_3$. Also $g$ is of class $C^1$, hence so is $\widetilde{g}$. Since the functions $\widetilde{g}$ and $\hat{g}$ agree on $E_{j_0}$, we then also have that
$$
\nabla\hat{g}(x)=\nabla \widetilde{g}(x)$$
for almost every $x\in E_{j_0}$ (see \cite[Theorem 3.3(i)]{EvansGariepy} for instance).

Next, by applying Alexandroff's theorem \cite{Alexandroff} (see also \cite{BusemannFeller} in dimension $2$) with the convex function $\hat{g}$, we obtain that $\hat{g}$ is twice differentiable almost everywhere in $\Omega_3$. This implies that
\begin{multline}\label{equality of second order limits for f tilde and f hat}
\limsup_{ y\rightarrow x,\;y\in
  E_{j_0}} \frac{|\widetilde{g}(y)-\widetilde{g}(x)-\langle \nabla
  \widetilde{g}(x),y-x\rangle|}{ |y-x|^2}= \\ =\limsup_{
  y\rightarrow x,\;y\in E_{j_0}} \frac{|\hat{g}(y)-\hat{g}(x)-\langle \nabla \hat{g}(x),y-x\rangle|}{ |y-x|^2}<+\infty
\end{multline}
for almost every $x\in E_{j_0}$. However, by the definition of $\widetilde{g}(x)=g(x)+j_0|x|^2$, we have
\begin{multline*}
\frac{|g(y)-g(x)-\langle \nabla g(x),y-x\rangle|}{|y-x|^2}\leq \\ \leq
\frac{|g(y)-g(x)-\langle \nabla g(x),y-x\rangle+(j_0(|y|^2+|x|^2-2\langle x,y\rangle)|}{ |y-x|^2} +j_0= \\
= \frac{|\widetilde{g}(y)-\widetilde{g}(x)-\langle \nabla \widetilde{g}(x),y-x\rangle|}{|y-x|^2}+j_0,
\end{multline*}
and by combining with \eqref{equality of second order limits for f tilde and f hat} we immediately obtain Claim \eqref{4}.
\end{proof}

\begin{cor}
Let $\Omega\subset\R^n$ be a Lebesgue measurable set, and $f:\Omega\rightarrow\R$ be a function such that for almost every $x\in\Omega$ there exists a vector $\xi_x\in\R^n$ such that
\begin{equation}\label{proximal limit1}
\limsup_{y\rightarrow x}\frac{f(y)-f(x)-\langle\xi_x,y-x\rangle}{ |y-x|^2}<+\infty.
\end{equation}
Then for every $\varepsilon>0$ there exists a function $g\in C^{2}(\R^n)$ such that
$$
\mathcal{L}^{n}\left(\{x\in\Omega : f(x)\neq g(x)\}\right)\leq\varepsilon.
$$
\end{cor}
\noindent This is of course an immediate consequence of Theorem \ref{Theorem for k=2} applied to $-f$.

According to Remark \ref{the gradients equal the subgradients ae}, we also have that
$$
\xi_{x}=\nabla g(x) 
$$
for almost every $x\in\Omega$ with  $f(x)=g(x)$.
 
\begin{cor}
Let $f:\R^n\to\R$ be a measurable function, and define $\Omega=\{x\in\R^n : \partial_{P}f(x)\neq\emptyset\}$. Then for every $\varepsilon>0$ there exists a function $g\in C^{2}(\R^n)$ such that
$$
\mathcal{L}^{n}\left(\{x\in\Omega : f(x)\neq g(x)\}\right)\leq\varepsilon.
$$
\end{cor}
Here $\partial_{P}f(x)$ denotes the proximal  subdifferential of $f$ at $x$, which is defined as the set of all $\zeta\in\R^n$ for which there exist $\sigma, \eta>0$ such that
$$
f\left( y \right) \geq f\left( x \right) + \left\langle {\zeta ,y
- x} \right\rangle  - \sigma |y - x|^2 
$$
for all $y \in B\left( {x,\eta } \right)$. The set $\partial_{P}f(x)$ coincides with $\{\zeta\in\R^n : \zeta=\nabla\varphi(x), \varphi\in C^2(\R^n), f-\varphi$ attains a minimum at $x\}$, so every function $f$ for which the viscosity subdifferential of second order is nonempty at $x$ also has a nonempty proximal subdifferential at $x$. The set $\partial_{P}f(x)$ can also be equivalently defined as the set of vectors $\zeta\in\R^n$ such that
$$
\liminf_{h\to 0}\frac{f(x+h)-f(x)-\langle \zeta, h\rangle}{|h|^2}>-\infty,
$$
so it is clear that the above Corollary is an immediate consequence of Theorem \ref{Theorem for k=2}. Notice also that this corollary allows us to recover, with a different proof, the mentioned result for convex functions established independently by Alberti \cite{Alberti2} and Imonkulov \cite{Imonkulov}.

\medskip

Let us finally present two examples. The first one concerns the following matter: one could erroneously think that if a function $f$ satisfies \eqref{proximal limit} then $f$ will automatically satisfy 
\begin{equation}\label{upper proximal limit}
\limsup_{y\rightarrow x}\frac{f(y)-f(x)-\langle\xi_x,y-x\rangle}{|y-x|^2}<+\infty
\end{equation}
for almost every $x\in\Omega$ as well, and then one could immediately apply Liu-Tai's theorem \cite{LiuTai} to conclude the proof of Theorem \ref{Theorem for k=2}. This is not feasible.
\begin{ex}
{\em Let us first consider a Cantor set of positive measure, $C\subset [0,1]$. More precisely,
$$
C=[0,1]\setminus \bigcup_n J_n
$$
where each $J_n$ is the union of $2^{n-1}$ disjoint intervals of length $\frac{1}{4^n}$ and 
$J_n\cap J_m=\emptyset$ for $n\not= m$. 
$$
J_n=\bigcup_{k=1}^{2^{n-1}}(a_n^k,b_n^k),
$$
where $b_n^k<a_n^{k+1}$ for $k<2^{n-1}$. Let us inductively construct
the sets $J_n$. Setting $J_1=(\frac{3}{8},\frac{5}{8})$, if $n\geq 1$, we assume that
$J_1,\dots ,J_n$ satisfy that
$$
[0,1]\setminus \bigcup_{k=1}^nJ_k
$$
consists in $2^n$ disjoint intervals of length $\frac{1}{2^{n+1}}+\frac{1}{2^{2n+1}}$,
because
$$
\mathcal{L} \bigl([0,1]\setminus \bigcup_{k=1}^nJ_k\bigr) =1-\sum_{k=1}^n \frac{2^{k-1}}{4^k}
=1-\frac{1}{2}(1-\frac{1}{2^n})=\frac{1}{2}+\frac{1}{2^{n+1}}.
$$
For each of these intervals composing $[0,1]\setminus \bigcup_{k=1}^nJ_k$, we consider a subinterval, centered at the corresponding middle point, of length $\frac{1}{4^{n+1}}$. Then
$J_{n+1}$ will be the union of these subintervals. It is clear that $\mathcal{L}(C)=\frac{1}{2}$.

Now let us define a function $f$ in the following way: we set
$$
f(x)=0 \ \ \hbox{for every} \ \ x\in C,
$$
while for every $n\in \N$ and $k=1,\dots ,2^{n-1}$, 
$f:[a_n^k,b_n^k]\to \mathbb{R}$ will be a non negative continuous function
such that $f:(a_n^k,b_n^k)\to \mathbb{R}$ is $C^{\infty}$, 
$$
\max _{x\in I_n^k}f(x)=f(a_n^k+\frac{1}{2}(b_n^k-a_n^k))=\frac{1}{2^n},
$$
and such that $f$, as well as all its one-sided derivatives,  equal $0$ at $a_n^k$ and at $b_n^k$.
It is clear that $f$ is continuous. Let us denote
$$
\Delta _x(y)=\frac{f(y)-f(x)-\xi _x(y-x)}{|y-x|^2}.
$$
If $x\not\in C$ then, taking $\xi _x=f'(x)$, we have $\lim_{y\to x}\Delta _x(y)=\frac{1}{2}f''(x)$. 
If $x\in C$, then
$$
\frac{f(y)-f(x)}{|y-x|^2}\geq 0.
$$
Hence for every $x$ there exists $\xi _x$ such that
$$
\liminf _{y\to x}\Delta _x(y)>-\infty.
$$
Let us observe that $f$ also satisfies conditions of the form
$$
\liminf_{y\to x}\frac{f(y)-P(y-x)}{|y-x|^k}>-\infty,
$$
where $P$ is a polynomial of degree $k-1$ for every $k$.  

Now let $\tilde{C}=C\setminus ( \{ 0,1\} \cup \{ a_n^k,b_n^k\} _{n,k})$.
We claim that 
$$
\limsup_{y\to x}\Delta _x(y)=+\infty
$$
for every $x\in \tilde{C}$ and every $\xi _x$. Let us prove this.
If $x\in \tilde{C}$ there exist subsequences  $\{ a_{m_j}^{r_j}\} _j$ 
and $\{ b_{n_j}^{k_j}\} _j$, decreasing and increasing respectively, such that
$$
 \lim_ja_{m_j}^{r_j} =\lim_jb_{n_j}^{k_j}=x.
$$
More precisely, we chose $a_{m_j}^{r_j}$ such that
$$
0<a_{m_j}^{r_j}-x \leq \frac{1}{2^{m_j+1}}+\frac{1}{2^{2m_j+1}}, 
$$
and $b_{n_j}^{k_j}$ such that
$$0<x- b_{n_j}^{k_j} \leq
\frac{1}{2^{n_j+1}}+\frac{1}{2^{2n_j+1}}.
$$
Let us consider the case that $\xi _x\geq 0$. 
We take $y_j=b_{n_j}^{k_j}-\frac{1}{2}(b_{n_j}^{k_j}-a_{n_j}^{k_j})$. We have
$$
\Delta _x(y_j)\geq \frac{f(y_j)}{|y_j-x|^2}=\frac{1}{2^{n_j}}\frac{1}{|y_j-x|^2}\geq 2^{n_j}
$$
since
$|y_j-x|\leq \frac{1}{2^{n_j}}$. In particular we obtain that $\limsup_{y\to x}\Delta _x(y)=+\infty$.

The case $\xi _x\leq 0$ can be dealt with similarly by considering
$y_j=a_{m_j}^{r_j}+\frac{1}{2}(b_{m_j}^{r_j}-a_{m_j}^{r_j})$. \qed 

}
\end{ex}

Our second example shows that there are no analogues of Theorem \ref{Theorem for k=2} for higher order of differentiability. 

\begin{ex}
{\em
  Let $f:\mathbb{R}\to \mathbb{R}$ be 
  the function given by
  \begin{equation*}
    f(x)=\frac{1}{\pi^2}\sum_{n=1}^\infty 2^{-3n}\cos\left( 2^{n} \pi x\right).
  \end{equation*}
  This is a $C^2$ function such that $f''$ is not differentiable at
  any point (see \cite{Hardy}) and 
  \begin{equation*}
    \limsup_{|y|\to 0} \frac{\left|f''(x+y)+f''(x-y)-2f''(x)\right|}{|y|}<+\infty
  \end{equation*} 
for every $x\in \mathbb{R}$ (see \cite[p. 148]{Stein}). By \cite[Theorem 4]{LiuTai} $f''$ is not approximately differentiable on a
set of positive measure.

  For every $x$, we have that
$$
\lim_{y\to x}\frac{f(y)-f(x)-f'(x)(y-x)-\frac{1}{2}f''(x)(y-x)^2}{|y-x|^2}=0
$$
If $a>0$ 
 we have 
$$
\liminf_{y\to x}\frac{f(y)-f(x)-f'(x)(y-x)-(\frac{1}{2}f''(x)-a)(y-x)^2}{|y-x|^2}>0,
 $$
hence
$$
\liminf_{y\to x}\frac{f(y)-f(x)-f'(x)(y-x)-(\frac{1}{2}f''(x)-a)(y-x)^2}{|y-x|^k}=\infty >-\infty
$$
for every $k>2$.
If an
analogue of Theorem \ref{Theorem for k=2} for some order $k>2$ were true
for this function $f$, then, according to Liu-Tai's characterization of Lusin properties and approximate differentiability of higher order \cite{LiuTai}, we would have that $f$ is approximately differentiable of order $k$. However, in \cite[p. 194]{LiuTai} it is shown that the coefficients of order $j$ of the Taylor expansion of an approximately differentiable function of order $k$
coincide, up to sets of arbitrarily small measure, with derivatives of order $j$ of $C^k$ functions; in particular those coefficients have the Lusin property of class $C^{k-j}$ and therefore, again by \cite[Theorem 1]{LiuTai}, they are almost everywhere approximately differentiable of order $k-j$. This would imply that $f''$ is approximately
differentiable almost everywhere, which we know to be false.

Another example can be given by taking $g:\R\to\R$ to be a continuous function which is nowhere approximately differentiable (see \cite[Chapter 6]{Kharazishvili}), setting $$f(x)=\int_{0}^{x}\left(\int_{0}^{t}g(s)ds\right)dt,$$ and repeating the preceding argument word by word. One could also use as $g$ the Takagi function of Example \ref{example Takagi}, which by \cite[Theorem 2]{BrownKozlowski} and \cite{LiuTai} is not approximately differentiable on any set of positive measure.
 }
\end{ex}

\end{document}